\documentclass[a4paper,12pt,reqno]{amsart}

\usepackage[T1]{fontenc}
\usepackage[utf8]{inputenc}

\usepackage[english]{babel}

\usepackage{amsmath, amssymb, amsthm, amsfonts}
\usepackage{mathrsfs}
\usepackage{tikz}
\usetikzlibrary{arrows}
\usetikzlibrary{matrix}
\usetikzlibrary{cd}
\usepackage[left=2.5cm, right=2.5cm, top=2cm]{geometry}
\usepackage{enumerate}
\usepackage{hyperref}
\usepackage[singlespacing]{setspace}

\theoremstyle{plain}
\newtheorem{thm}{Theorem}
\newtheorem{lem}{Lemma}
\newtheorem{cor}{Corollary}
\newtheorem{con}{Conjecture}
\newtheorem*{thm*}{Theorem}
\newtheorem{prop}{Proposition}

\theoremstyle{definition}
\newtheorem{defi}{Definition}
\newtheorem{rema}{Remark}

\newtheorem{ques}{Question}
\newtheorem*{ques*}{Question}

\newcommand{\Z}{\mathbb{Z}}
\newcommand{\Soc}{\textnormal{Soc}}
\newcommand{\Hom}{\textnormal{Hom}}
\newcommand{\Aut}{\textnormal{Aut}}

\begin{document}

\title{Braces of order $p^2q$}
\author{Carsten Dietzel}
\email{carstendietzel@gmx.de}
\address{Institute of algebra and number theory, University of Stuttgart, Pfaffenwaldring 57, 70569 Stuttgart, Germany}
\date{\today}

\begin{abstract}
In this article we classify the left braces of order $p^2q$ where $p,q$ are primes fulfilling $q > p+1$.

This classification includes a proof of three conjectures of Guarnieri and Vendramin (\cite[Conjectures 6.2-6.4]{Vendramin_skew}) concerning the number of isomorphism classes of left braces of order $p^2q$ for certain values of $p,q$.
\end{abstract}

\maketitle

\section*{Introduction}

This article will revolve around ring-like structures called \emph{braces} which under this name first occurred in \cite{Rump_braces_radical_rings}. But the \emph{concept} of brace is slightly older:

In Rump's works, the first occurrence of the concept is in \cite{Rump_decomposition_theorem}, then in the form of \emph{linear cycle sets} which he used to investigate set-theoretical solutions of the quantum Yang-Baxter equation which reads
\begin{equation}
R^{12}R^{13}R^{23} = R^{23}R^{13}R^{12} \tag{YB} \label{eq:yang_baxter}
\end{equation}
where $R:X \times X \to X \times X$ is a bijection and $X$ is a finite set. Studying \eqref{eq:yang_baxter} has been motivated by Drinfeld in \cite{drinfeld_quantum} in order to construct classes of solutions of the classical (i.e. linear) Yang-Baxter equation.

Rump's linear cycle sets, however, have another equivalent in the work of Etingof, Schedler and Soloviev who, in \cite{Etingof_sybe}, particularly showed that a broad class of set-theoretical solutions of the Yang-Baxter equation can be constructed from \emph{bijective 1-cocycles}.

After having dropped the words \emph{brace}, \emph{linear cycle set}, and \emph{bijective 1-cocycle} we will now explain what they mean and sketch the interconnections between the respective objects. We will introduce these as generalizations of a more familiar class of algebraic objects, i.e. the \emph{radical rings} (\cite{jacobson}):

A \emph{radical ring} $(A,+,\cdot)$ is a ring - not necessarily with $1$ - which coincides with its own Jacobson radical, or, equivalently, which under the \emph{Jacobson circle operation} given by
\begin{equation} \label{eq:jacobson_circle}
a \circ b := a + b + a \cdot b \tag{J}
\end{equation}
becomes a group. Therefore, a radical ring always carries two group structures, namely, its common additive group $(A,+)$ and its \emph{adjoint group} given by $(A,\circ)$.

Repeating the ring axioms, one might thus define a radical ring as an abelian group $(A,+)$ with a left- and right-distributive multiplication $\cdot$ such that $A$ also becomes a group under the circle operation given by Eq. \eqref{eq:jacobson_circle}.

If the right-distributive law is dropped we arrive at the notion of \emph{left brace}, which consequently is an abelian group with a left-distributive multiplication $\cdot$ such that $(A,\circ)$ defines a group structure on $A$ - similarly, one defines a \emph{right brace}.

It can be shown that a left brace can alternatively be defined as an abelian group $(A,+)$ with a binary operation $\cdot: A \times A \to A$ which fulfils the following three axioms
\begin{align}
a \cdot (b + c) & = a \cdot b + a \cdot c \tag{B1} \\
(a \cdot b + a + b) \cdot c & = a \cdot (b \cdot c) + a \cdot c + b \cdot c  \tag{B2} \\
x \mapsto a \cdot x + x & \textnormal{ is bijective for each $a \in A$} \tag{B3}.
\end{align}
The circle operation on $A$ is defined as above and defines a second group structure on $A$.

Radical rings can therefore be regarded as left braces which additionally fulfil the axiom
\begin{equation}
(a + b) \cdot c = a \cdot c + b \cdot c \tag{B4}.
\end{equation}

From now on we denote the additive group of a left brace as $A^+$ and the respective adjoint group as $A^{\circ}$.

The interplay between the groups $A^{\circ}$ and $A^+$ is given by the \emph{adjoint action}
\begin{align*}
\bullet: A^{\circ} \times A^+ & \to A^+ \\
(a,x) & \mapsto a \bullet x := a \cdot x + x 
\end{align*}
and it turns out that $\bullet$ makes $A^+$ a (left) $A^{\circ}$-module such that the identity map $\pi: A \to A$ induces a bijective $1$-cocycle $A^{\circ} \to A^+$ for this action, i.e. we have for $g,h \in A^{\circ}$ the identity
\[
\pi(gh) = \pi(g) + g \bullet \pi(h).
\]
Vice versa, an abelian group $A$ which is a left module over a group $G$ becomes a brace by giving a bijective $1$-cocycle $\pi:G \to A$: if the module action is given by $\bullet$, defining on $A$ the binary operation $a \cdot b := \pi^{-1}(a) \bullet b - b$ and leaving the addition unchanged makes $A$ a left brace.

Another equivalent notion is the one of a \emph{linear cycle set}, that is, an abelian group $A$ with a multiplication map $\cdot$ which fulfils the axioms
\begin{align}
a \cdot (b + c) & = a \cdot b + a \cdot c \tag{L1} \\
(a + b) \cdot c & = (a \cdot b) \cdot (a \cdot c) \tag{L2}. 
\end{align}
Defining ${}^ba$  as the inverse of $a$ under the right multiplication map $x \mapsto x \cdot b$, one can then associate with $A$ a left brace by the identity $b \circ a = {}^ba + b$, $\circ$ being the adjoint group operation.

All of the above equivalences can be found, with proofs, in \cite{Rump_cyclic_braces}.

The applications of left braces resp. bijective $1$-cocycles resp. linear cycle sets are manifold: a few examples are the theory of Bieberbach groups, orderable groups and solutions of the SYBE. A good overview over these and other occurrences of braces in different mathematical disciplines is given in the survey article \cite{Rump_brace_classical}.

After the advent of left braces, many variations of these structures have been investigated - two examples are \emph{Hopf braces} (\cite{Vendramin_hopf}) and \emph{skew left braces} (\cite{Vendramin_skew}).

It is natural to ask for a classification of \emph{all} left braces up to isomorphism. However, this question might be too general to be answered in a satisfying way.

Better questions to ask are the following:

\begin{ques}
Classify / enumerate the isomorphism classes of left braces of order $n$.
\end{ques}

\begin{ques}
Classify / enumerate the isomorphism classes of left braces with given additive / adjoint group.
\end{ques}

Concerning the classification aspect, the second question has recently been completely answered by Rump for $A^+$ being a cyclic group, see \cite{Rump_cyclic_braces} and \cite{Rump_cyclic_2}.

Here, we are mainly interested in the first question:

We note two achievements which will be of relevance for this article:

\begin{itemize}
\item Bachiller, in \cite{Bachiller_classification}, classifies all left braces of orders $n = p^2,p^3$ where $p$ is a prime.

\item Guarnieri and Vendramin, in \cite{Vendramin_skew}, use computational methods to determine for most $n \leq 120$ (except for $n \in \{32, 64,81,96 \}$) the numbers $b(n)$, being defined as the number of non-isomorphic braces of fixed order $n$.

See \cite[p.2530]{Vendramin_skew} for the results.
\end{itemize}

From their computations they derive the following conjectures about the behaviour of $b(n)$:

\begin{con}(\cite[Conjecture 6.2]{Vendramin_skew})
If $q > 3$ is prime then
\[
b(4q) =
\begin{cases}
11 & q \equiv 1 \mod 4 \\
9 & q \equiv 3 \mod 4
\end{cases}
\]
\end{con}

\begin{con}(\cite[Conjecture 6.3]{Vendramin_skew})
If $q > 3$ is prime then\footnote{The reader may have noted that the original statement of the conjecture did not include the case $q \equiv 8 \mod 9$; the reason is that the first occurrence of this case is $n = 9 \cdot 17 = 153$ which is not covered by the computations of Guarnieri and Vendramin.}
\[
b(9q) =
\begin{cases}
14 & q \equiv 1 \mod 9 \\
4 & q \equiv 2,5,8 \mod 9 \\
11 & q \equiv 4,7 \mod 9
\end{cases}
\]
\end{con}

\begin{con}(\cite[Conjecture 6.4]{Vendramin_skew})
If $p,q$ are primes such that $p < q$ and $p \nmid q-1$ then $b(p^2q) = 4$.
\end{con}

The last of these conjectures has already been proved by Smoktunowicz (\cite{Smoktunowicz_note_on_sybe}).

In this article we will, more generally, prove these three conjectures by giving a complete classification of the non-isomorphic braces of order $p^2q$ with primes $p,q$ fulfilling $q > p+1$.

We define in \autoref{sec:generalities_braces} braces and modules over braces and introduce a special case of Rump's construction of semidirect products (see \cite{Rump_semidirect}) - if $B$ is an abelian group which is a module over a brace $A$ then there is a semidirect product $A \ltimes B$ coming from this action. This \emph{external} notion of semidirect product will be given an \emph{internal} counterpart.

Making $B$ an $A$-module will turn out to be the same as giving a group homomorphism from $A^{\circ}$ to $\Aut(B)$ (or, simply, making $B$ an $A^{\circ}$-module in the common sense).

In \autoref{sec:structure_result}, we show that every brace of order $p^2q$ ($q > p+1$) is a semidirect product $A_p \ltimes A_q$ - ($A_p, A_q$ being the primary components of the additive group) as described above. Therefore, an exhaustive list of all of these braces can be given by describing all possible $A_p$-module structures on $A_q$.

We will show that if $A_p$ is fixed then the isomorphism classes of semidirect products $A_p \ltimes A_q$ are in a bijective correspondence with the orbits of $\Hom(A_p^{\circ},\Z_q^{\times})$ under the action of $\Aut(A_p)$ (meaning brace automorphisms).

In \autoref{sec:case_p_not_2} and \autoref{sec_case_p_2} this plan will then be implemented and it will essentially turn out that each module structure comes from an action on $A_q$ by $p$-th or $p^2$-th roots of unity $\mod q$.

The final classification results will consequently be dependent of the residue class of $q \mod p^2$.

These results will confirm Guarnieri's and Vendramin's conjecture for $p=2$ and will generalize them for $p \neq 2$ showing that
\[
b(p^2q) =
\begin{cases}
4 & p \nmid q-1 \\
p+8 & p \mid q-1, p^2 \nmid q-1 \\
2p+8 & p^2 \mid q-1.
\end{cases}
\]
which clearly includes conjecture 6.3 as the case $p=3$.

\section{Generalities on braces} \label{sec:generalities_braces}

\begin{defi}
A \emph{left brace} is an abelian group $A$ - the addition denoted here by $+$ - together with a multiplication  $\cdot $ such that the following three axioms are fulfilled:
\begin{align}
a \cdot (b + c) & = a\cdot b + a \cdot c \tag{B1} \label{eq:brace_axiom_1}\\
(a \cdot b + a + b) \cdot c & = a \cdot (b \cdot c) + a \cdot c + b \cdot c \tag{B2} \label{eq:brace_axiom_2} \\
x \mapsto a \cdot x + x & \textnormal{ is bijective for each $a \in A$} \label{eq:brace_axiom_3} \tag{B3}
\end{align}
\end{defi}

It should be clear how \emph{homomorphisms} and \emph{isomorphisms} between braces are defined.

Throughout this article we will always assume $A$ to be finite.

Additionally to its additive structure, a brace always carries a second group structure:

\begin{prop} \label{prop:group_under_circle}
If $A$ is a brace, the operation $a \circ b := a + b + a \cdot b$ defines a group structure on $A$.
\end{prop}

\begin{proof}
The proof can be found under \cite[Proposition 4]{Rump_braces_radical_rings} but for the reader's convenience, we reprove the statement here:

Denoting the neutral element of $A^+$ by $0$, we claim that this is also a neutral element for $\circ$. Indeed, Eq. \eqref{eq:brace_axiom_1} directly implies $a\cdot 0$ for all $a$, and therefore
\[
a \circ 0 = a + 0 + a \cdot 0 = a.
\]
Setting $a = b = 0$ in Eq. \eqref{eq:brace_axiom_2} results in $0 \cdot ( 0 \cdot c) + 0 \cdot c = 0$.

The equation $0 \cdot x + x = 0$ has, by \eqref{eq:brace_axiom_3}, a unique solution. By the above reasoning, $x = 0 \cdot c$ is such a solution, as is $x = 0$. We conclude $0 \cdot c = 0$. Therefore,
\[
0 \circ a = 0 + a + 0 \cdot a = a.
\]

Associativity follows from the calculation
\begin{align*}
a \circ (b \circ c) & = a \circ (b + c + b \cdot c) \\
& = a + b + c + b \cdot c + a \cdot ( b + c + b \cdot c) \\
& = a + b + c + b \cdot c + a \cdot b + a \cdot c + a \cdot (b \cdot c) \\
& \overset{(\textnormal{\ref{eq:brace_axiom_2}})}{=} a + b + a \cdot b + c + (a \cdot b + a + b) \cdot c \\
& = (a + b + a \cdot b) \circ c \\
& = (a \circ b) \circ c.
\end{align*}

From \eqref{eq:brace_axiom_3}, for each $a$ there is an $x$ fulfilling $a \cdot x + x = -a$. This $x$ also fulfils
\[
a \circ x = a + x + a \cdot x = a + (-a) = 0.
\]
$x$ is a right-inverse which consequently has to be a left-inverse, too.
\end{proof}

\begin{defi}
Let $A$ be a brace. The $\emph{adjoint group}$ $A^{\circ}$ is then defined as the group with the underlying set $A$, together with the group operation $a \circ b = a + b + a \cdot b$.
\end{defi}

The interplay between the groups $(A,+)$ and $(A,\circ)$ is as follows:

\begin{prop} \label{prop:braces_are_1_cocycles}
Setting $a \bullet x := a \cdot x + x$ defines an action (from the left) of $A^{\circ}$ on $A$. Furthermore, the identity map $A^{\circ} \to A$ is a bijective $1$-cocycle for this action.
\end{prop}

\begin{proof}
We have $0x = 0 \cdot x + x = x$ and
\begin{align*}
(a \circ b) \bullet x & = (a + b + a \cdot b)\cdot x + x \\
& \overset{\textnormal{(\ref{eq:brace_axiom_2})}}{=} a \cdot x + b \cdot x + a \cdot (b \cdot x) + x \\
& = a \cdot (b \cdot x + x) + b \cdot x + x \\
& = a \cdot (b \bullet x) + (b \bullet x) = a \bullet (b \bullet x).
\end{align*}
Furthermore, the cocycle property for the identical map follows from the calculation
\[
\textnormal{id}(a \circ b) = a + b + a \cdot b = a + a \bullet b = \textnormal{id}(a) + a \bullet  \textnormal{id}(b).
\]
\end{proof}

For each abelian group $A$ there is a canonical way of making $A$ into a brace, i.e. by defining $a \cdot b = 0$ for any $a,b \in A$. The brace axioms are easily checked for this operation and it turns out that the underlying abelian group structure coincides with the adjoint group structure.

We will refer to the brace thus constructed as the \emph{trivial} brace associated with $A$. Consequently, we will call any brace fulfilling $a \cdot b = 0$ for each $a,b \in A$ a \emph{trivial} brace.

Clearly, this is the unique brace structure on the additive group $A$ with trivial adjoint action.

Due to the presence of both an addition as a multiplication there is a natural notion of \emph{left} and \emph{right ideal}:

\begin{defi}
Let $A$ be a brace. A \emph{left ideal} is an additive subgroup $B \leq A$ such that $aB \subseteq B$ holds for all $a \in A$. Similarly, a \emph{right ideal} is defined.

A subgroup is called an \emph{ideal} if it is a left and right ideal.
\end{defi}

It is not hard to see that each left resp. right ideal is also a subbrace.

The left ideals of $A$ can furthermore be specified as the additive subgroups $B$ of $A$ which are invariant under the adjoint action:

A left ideal $B$ fulfils for each $a \in A$
\[
a \bullet B \subseteq a \cdot B + B \subseteq B + B = B
\]
and if $B$ is a subgroup invariant under the adjoint operation then
\[
a \cdot B \subseteq a \bullet B - B \subseteq B - B = B
\]
thus proving what we claimed.

\begin{prop} \label{prop:socle_is_subbrace}
If $A$ is a brace then the set
\[
\Soc(A) := \{a \in A \vert \forall b \in A: a \cdot b = 0 \}
\]
is an ideal which, as a subbrace, is trivial.
\end{prop}

\begin{proof}
For a proof that $\Soc(A)$ is an ideal we refer to \cite[Proposition 7]{Rump_braces_radical_rings}.

The triviality assumption, however, follows immediately from the definition of $\Soc(A)$.
\end{proof}

\begin{defi}
The subset $\Soc(A)$ defined in $\autoref{prop:socle_is_subbrace}$ is called the \emph{socle} of $A$.
\end{defi}

It is easily to see that $\Soc(A)$ can also be identified as the kernel of the homomorphism $\rho: A^{\circ} \to \Aut(A^+)$ induced by the adjoint action, i.e. it consists exactly of those elements which act trivially by the adjoint action.

Moreover, the socle of a brace has the property that it can be factored out, thus leading to the notion of \emph{retraction} of a brace (see \cite[Sections 1,2]{Rump_braces_radical_rings} )
But for technical reasons we will look at $\Soc(A)$ as a special case of a larger class of trivial subbraces which we will call \emph{trivial}:

\begin{defi}
An ideal $B \leq A$ is called \emph{trivial} if it is contained in $\Soc(A)$.
\end{defi}

The trivial ideals enjoy the following useful property:

\begin{lem}
Let $A$ be a brace. Then the trivial ideals of $A$ coincide with the additive subgroups of $A$ which are contained in $\Soc(A)$ and which are normal subgroups of $A^{\circ}$.
\end{lem}

\begin{proof}
First of all, it should be noted that the triviality of $\Soc(A)$ implies that any additive subgroup of $A$ contained in $\Soc(A)$ is also a subgroup of $A^{\circ}$ and vice versa.

Any additive subgroup $B$ contained in $\Soc(A)$ trivially fulfils $B\cdot a = \{ 0 \} \subseteq B$. On the other hand, we have the equivalences
\begin{align*}
a \circ B & = B \circ a  \\
\Leftrightarrow  a + a \bullet B & = B + \underbrace{B \bullet a}_{=a} \\
\Leftrightarrow a \bullet B & = B,
\end{align*}
which implies that $B$ is normal in $A^{\circ}$ iff $B$ is invariant under the adjoint action. We have already given an argument that the latter condition is equivalent to $B$ being a left ideal, i.e. an ideal.
\end{proof}

\begin{lem}
Let $A$ be a brace and $B \leq A$ a trivial subbrace. Then the adjoint action of $A$ on the elements of $B$ coincides with conjugation action of $A^{\circ}$ on $B^{\circ}$, i.e.:
\[
{}^{a\circ} b = a \bullet b.
\]
Especially, $B$ is in this case stable under the adjoint action.
\end{lem}

\begin{proof}
Clearly, ${}^{a\circ} b \circ a = a \circ b$. Writing the definition of $\circ$ out gives us
\begin{align*}
{}^{a\circ} b + a + \underbrace{{}^{a\circ} b \cdot a}_{=0} & = a + b + a \cdot b \\
\Leftrightarrow \quad {}^{a\circ} b & = b + a \cdot b \\
& = a \bullet b
\end{align*}
\end{proof}

\begin{defi} \label{def:semidirect_product_intern}
Let $A$ be a brace with a direct sum composition of the additive group into left ideals $A = B \oplus C$. If $C$ is a trivial ideal we call $A$ an (internal) \emph{semidirect product} of $B$ and $C$ and write $A = B \ltimes C$.
\end{defi}

As in group theory, it is  possible to \emph{externalize} semidirect products, i.e. to describe the interplay between the factors $B,C$ by making $C$ a $B$-module. This can essentially be regarded as a special semidirect products of braces (see \cite[p.480]{Rump_semidirect}):

\begin{defi}
Let $B$ be a brace and $C$ an abelian group. A binary operation
\begin{align*}
B \times C & \to C \\
(b,c) & \mapsto b \cdot c
\end{align*}
makes $C$ a \emph{$B$-module} if the following axioms hold:
\begin{align}
0 \cdot x & = 0 \tag{M1} \label{eq:brace_module_1} \\
a \cdot (x + y) & = a \cdot x + a \cdot y \tag{M2} \label{eq:brace_module_2} \\
(a + b + a \cdot b) \cdot x & = a \cdot x + b \cdot x + a \cdot (b \cdot x) \tag{M3} \label{eq:brace_module_3}
\end{align}
\end{defi}

\begin{prop} \label{prop:modules_are_circ_modules}
A $B$-module $C$ can be regarded as a $B^{\circ}$-module under the operation $a \bullet x = a \cdot x + x$.

On the other hand, any $B^{\circ}$-module $C$ can be made into a $B$-module by setting $b \cdot x = b \bullet x - x$.
\end{prop}

\begin{proof}
The proof of the first part is syntactically the same as the proof of \autoref{prop:group_under_circle}.

For the second statement we only prove the harder part, i.e. Eq. \eqref{eq:brace_module_3} which can be rewritten as
\begin{align*}
(a \circ b) \bullet x - x & = a \bullet x - x + b \bullet x - x + a \bullet (b \bullet x - x) - (b \bullet x - x) \\
\Leftrightarrow a \bullet (b \bullet x) - x & = a \bullet x - x + b \bullet x - x + a \bullet (b \bullet x) - a \bullet x - b \bullet x + x \quad \checkmark
\end{align*}
Checking the remaining axioms is left to the reader.
\end{proof}

\begin{defi} \label{def:semidirect_product}
Let $A$ be a brace and $B$ be an $A$-module. We then define the (external) semidirect product $A \ltimes B$ as follows:

Take for the additive group the group $A \oplus B$ and define a multiplication by
\[
(a_1,b_1)\cdot (a_2,b_2) := (a_1 \cdot a_2,a_1 \cdot b_2).
\]
\end{defi}

\begin{prop}
The construction of \autoref{def:semidirect_product} results in a brace.
\end{prop}

\begin{proof}
Eqs. \eqref{eq:brace_axiom_1}, \eqref{eq:brace_axiom_2} can be deduced from Eqs. \eqref{eq:brace_module_2}, \eqref{eq:brace_module_3} by an easy calculation.

On the other hand, \eqref{eq:brace_axiom_3} follows from the calculation
\[
(a_1,b_1) \bullet (a_2,b_2) = (a_1,b_1) \cdot (a_2,b_2) + (a_2,b_2) = (a_1 \bullet a_2, a_1 \bullet b_2)
\]
which proves the bijectivity of the maps $(a_2,b_2) \mapsto (a_1,b_1) \cdot (a_2,b_2) + (a_2,b_2)$.
\end{proof}

As in group theory, the internal and external notions of semidirect product coincide:

\begin{prop} \label{prop:internal_external}
If $A = B \ltimes C$ is an internal semidirect decomposition then there is a natural $B$-module structure on $C$ such that $B \ltimes C$ (externally) is isomorphic to $A$.

Vice versa, each external semidirect product $B \ltimes C$ naturally decomposes as an internal semidirect product.
\end{prop}

\begin{proof}
If $A = B \ltimes C$ holds internally then $C$ is, by definition, stable under the adjoint action of $A$, especially it can be regarded - by restriction - as a $B$-module.

Now decompose arbitrary $a_1,a_2 \in A$ uniquely as $a_i = b_i + c_i$ ($i = 1,2$).

We then calculate
\begin{align*}
a_1 \cdot a_2 & = (b_1 + c_1) \cdot (b_2 + c_2) \\
& = (b_1 + c_1 + \underbrace{c_1 \cdot b_1}_{=0}) \cdot (b_2 + c_2) \\
& = b_1 \cdot (c_1 \cdot (b_2 + c_2)) + b_1 \cdot (b_2 + c_2) + c_1 \cdot (b_2 + c_2) \\
& = b_1 \cdot 0 + b_1 \cdot b_2 + b_1 \cdot c_2 + 0 \\
& = b_1 \cdot b_2 + b_1 \cdot c_2
\end{align*}
and this is the unique additive decomposition in $B \oplus C$.

For the other direction is very easy we just sketch the idea: identify $B$ with the elements of the form $(b,0)$ and $C$ with the elements $(0,c)$ in $B \ltimes C$. Then one can quickly calculate that $B,C$ are ideals and $C$ is trivial.
\end{proof}

\begin{defi}
Let $A$ be a brace and $\varphi:A^+ \to A^+$ an automorphism of the additive group. Then $A^{\varphi}$ is defined as the additive group of $A$ with the multiplication
\[
a \cdot_{\varphi} b := \varphi^{-1}(\varphi(a) \cdot \varphi(b))
\]
\end{defi}

Directly from this definition we get
\begin{cor} \label{cor:twisting_braces_by_auto}
\begin{enumerate}[a)]
\item $A^{\varphi}$ is a brace,
\item Let $A_1$,$A_2$ be two braces with the same underlying additive group and $\varphi:A \to A$ an additive automorphism. Then $\varphi: A_1 \to A_2$ is an isomorphism of braces iff $A_2^{\varphi} = A_1$,
\item if $\varphi_A:A^+ \to A^+$ and $\varphi_B:B^+ \to B^+$ are additive homomorphisms, we get
\[
(A \ltimes B)^{(\varphi_A,\varphi_B)} = A^{\varphi} \ltimes_{\varphi} B
\]
where the module structure on $B$ is altered to $a \cdot_{\varphi} b := \varphi_B^{-1}(\varphi_A(a)\cdot \varphi_B(b))$.
\end{enumerate}
\end{cor}

For any brace - remember that our braces here are always assumed to be finite - there is a canonical decomposition (see \cite[Section 5]{Rump_braces_radical_rings})
\[
A = \bigoplus_{p}A_p
\]
of the additive group into $p$-primary components. Each $A_p$ is a left ideal (for the $p$-primary components are invariant under the adjoint action - in fact, they are invariant under each automorphism of $A$) and thus a subbrace.

Caution is demanded for the $A_p$ need not be right ideals!

Note that this decomposition also picks for each $p$ a special $p$-Sylow subgroup of $A^{\circ}$ - even if there is more than one for a fixed $p$.

In the following section this decomposition will turn out to be a \emph{semidirect} decomposition for braces of order $p^2q$ where $q > p+1$.

\section{Braces of order \texorpdfstring{$p^2q$}{p2q}: A general structure result} \label{sec:structure_result}

For $n \in \Z^+$ we denote by $b(n)$ the number of isomorphism classes of left braces of order $n$.

The aim of this article is a proof of the following conjectures of Guarnieri and Vendramin regarding some values of $b(n)$:

\begin{con}(\cite[Conjecture 6.2]{Vendramin_skew})
If $q > 3$ is prime then
\[
b(4q) =
\begin{cases}
11 & q \equiv 1 \mod 4 \\
9 & q \equiv 3 \mod 4
\end{cases}
\]
\end{con}

\begin{con}(\cite[Conjecture 6.3]{Vendramin_skew})
If $q > 3$ is prime then
\[
b(9q) =
\begin{cases}
14 & q \equiv 1 \mod 9 \\
4 & q \equiv 2,5,8 \mod 9 \\
11 & q \equiv 4,7 \mod 9
\end{cases}
\]
\end{con}

\begin{con}(\cite[Conjecture 6.4]{Vendramin_skew})
If $p,q$ are primes such that $p < q$ and $p \nmid q-1$ then $b(p^2q) = 4$.
\end{con}

These conjectures will be proved as special cases of a more general expression for $b(p^2q)$ where $p,q$ are primes such that $q > p+1$. See \autoref{cor:counting_p_not_2} and \autoref{cor:counting_p_2} for the results.

We cite, without proof, the following classification result of Bachiller (\cite[Proposition 2.4]{Bachiller_classification}). The following results of this article will strongly depend on it. Because it makes the formulation of the following results more comfortable, each of these braces will get an easily recognizable name:

\begin{thm}[Bachiller] \label{thm:bachiller_classification}
Up to isomorphism, the braces of order $n = p^2$ are given as follows:
\begin{enumerate}[a)]
\item The trivial brace $T^{p,p}$ associated with $\Z_p^2$,
\item the trivial brace $T^{p^2}$ associated with $\Z_{p^2}$,
\item $B^{p,p}:= \Z_p^2$, with the multiplication
\[
\begin{pmatrix}
x_1 \\ y_1
\end{pmatrix} \cdot 
\begin{pmatrix}
x_2 \\ y_2
\end{pmatrix} =
\begin{pmatrix}
y_1 y_2 \\ 0
\end{pmatrix}
\]
\item $B^{p^2}:= \Z_{p^2}$, with the multiplication
\[
x_1 \cdot x_2 = px_1x_2
\]
\end{enumerate}
\end{thm}

For our classification we need to determine the automorphism groups of the braces listed above:

\begin{lem} \label{lem:automorphism_groups_p2}
The automorphism groups of the braces of order $p^2$ are given by:
\begin{enumerate}[a)]
\item $\Aut(T^{p,p}) = GL_2(\Z_p)$,
\item $\Aut(T^{p^2}) = \Z_{p^2}^{\times}$,
\item $\Aut(B^{p,p}) = \left\{ 
\begin{pmatrix}
a & b \\ 0 & d
\end{pmatrix} \vline a = d^2
\right\} \leq GL_2(\Z_p)$
\item $\Aut(B^{p^2}) = 1+p\Z_{p^2} \leq \Z_{p^2}^{\times}$.
\end{enumerate}
\end{lem}

\begin{proof}
Only for $B^{p,p}$ and $B^{p^2}$, some reasoning is necessary:

Clearly, any automorphism of $B^{p,p}$ is given by a left-multiplication by an invertible matrix $\begin{pmatrix}
a & b \\ c & d
\end{pmatrix}$, i.e.
\[
\varphi \begin{pmatrix}
x \\ y
\end{pmatrix} = \begin{pmatrix}
a & b \\ c & d
\end{pmatrix} \begin{pmatrix}
x \\ y
\end{pmatrix}
\]
Writing out the equation $\varphi\begin{pmatrix}
x_1 \\ y_1
\end{pmatrix} \cdot \varphi \begin{pmatrix}
x_2 \\ y_2
\end{pmatrix} = \varphi \left( 
\begin{pmatrix}
x_1 \\ y_1
\end{pmatrix} \cdot
\begin{pmatrix}
x_2 \\ y_2
\end{pmatrix}
\right)$ gives us to the two equations
\begin{align}
(cx_1 + dy_1)(cx_2 + dy_2) & = ay_1y_2 \label{eq:auto1} \\
cy_1 y_2 & = 0 \label{eq:auto2}
\end{align}
Using Eq. \eqref{eq:auto2} results in $c= 0$. Inserting this result in Eq. \eqref{eq:auto1} finally gives $a = d^2$.

For $B^{p^2}$ each automorphism is given by right-multiplication by a unit $a \in \Z_{p^2}^{\times}$. Here the automorphism property reads
\[
p(ax_1)(ax_2) = apx_1x_2 \Leftrightarrow apx_1x_2 = px_1x_2,
\]
i.e. $a$ must fix the subgroup $p\Z_{p^2}$. This is the case exactly when $a \in 1+p\Z_{p^2}$.
\end{proof}

Now we will point out some general lemmata concerning the structure of some braces of order $p^2q$:

\begin{lem} \label{lem:aq_is_trivial}
Let $A$ be a brace of order $p^2q$ where $q > p+1$. If $A_q$ denotes the $q$-component of $A_q^+$ then $A_q$ becomes a trivial subbrace by restriction.

Furthermore, $A_q^{\circ}$ is the unique $q$-Sylow subgroup of $A^{\circ}$.
\end{lem}

\begin{proof}
We will make use of the adjoint operation of $A^{\circ}$ on $A^+$:

Clearly, $A^+ = A_p^+ \oplus A_q^+$, and it follows that
\[
\Aut(A^+) \cong \Aut(A^+_p) \times \Aut(A^+_q).
\]
$A_q^+$ must be isomorphic to $\Z_q$ which implies $\vert \Aut(A^+_q)\vert = q-1$.

Furthermore, $A_p^+$ is isomorphic to either $\Z_{p^2}$ or $\Z_p^2$. In the first case holds $\vert \Aut(A_p^+) \vert = p(p-1)$, whereas $\vert \Aut(A_p^+) \vert = (p^2 -1)(p^2 -p) = (p+1)p(p-1)^2$ holds in the second case.

We conclude that in each case we have
\[
\vert \Aut(A^+) \vert \quad \mid \quad (q-1)(p+1)p(p-1)^2.
\]
Any prime factor of $\vert \Aut(A^+) \vert$ is therefore $< q$ such that $\Aut(A^+)$ can not have a subgroup of order $q$.

Letting $A_q^{\circ}$ be any $q$-Sylow subgroup of $A^{\circ}$, it follows that $A_q^{\circ}$ acts trivially under the adjoint action, i.e. is contained in $\Soc(A)$.

This makes $A_q^{\circ}$ also a subgroup of $A^+$ of order $q$ for the additive and multiplicative structures coincide inside $\Soc(A)$.

Therefore, $A_q^{\circ} = A_q^+$, so $A_q^{\circ}$ is the only $q$-Sylow subgroup of $A^{\circ}$ which shows $A_q^{\circ} \unlhd A^{\circ}$.

$A_q^+$ is therefore a trivial subbrace of $A$.
\end{proof}

\begin{rema}
It should be noted that the proof of \autoref{lem:aq_is_trivial} implies that the condition $q > p+1$ actually is way to strong - if suffices to assume that $q$ divides neither of the three integers $p-1,p,p+1$ for the classification theorem to hold.

We just decided to use the stronger condition $q > p+1$ in our statements in order not to deviate too much from the original statements of Guarnieri's and Vendramin's conjectures.
\end{rema}

We can now deduce the following structure theorem:

\begin{thm} \label{thm:p2q_braces_are_semidirect}
Any brace $A$ of order $n = p^2q$ ($q > p+1 $) is a semidirect product $A_p \ltimes A_q$ in the sense of \autoref{def:semidirect_product}.
\end{thm}

\begin{proof}
We have the canonical decomposition $A = A_p \oplus A_q$ into left ideals which has been introduced at the end of \autoref{sec:generalities_braces}. \autoref{lem:aq_is_trivial} shows that this decomposition is an internal semidirect product, i.e. $A = A_p \ltimes A_q$. The theorem then follows from \autoref{prop:internal_external}.
\end{proof}

This gives us an implicit instruction how to proceed: we need to describe \emph{all} possible ways of making $\Z_q$ into an $A_p$-module where $A_p$ is a brace of order $p^2$.

By \autoref{prop:modules_are_circ_modules} this is equivalent to giving suitable group homomorphisms $A_p^{\circ} \to \Aut(\Z_q)$.

\begin{lem}
Let $A$  be a brace of order $p^2q$ ($q > p + 1$), and $A = A_p \oplus A_q$ its decomposition into primary components. Furthermore, let $(\varphi_p,\varphi_q):A_p \oplus A_q \to A_p \oplus A_q$ be an (additive) automorphism.

Then the braces $A^{(\varphi_p,\varphi_q)}, A^{(\varphi_p,id_{A_q})}$ are identical.
\end{lem}

\begin{proof}
$\varphi_q$ and $\varphi_q^{-1}$ can be represented as multiplications by constants $k,l \in \Z^+$ where $kl \equiv 1 \mod q$.

If we denote by $\cdot_{\varphi}$ the multiplication in $A^{\varphi}$, then in the semidirect product decomposition $A_p \ltimes A_q$, the module multiplication is altered by $\varphi = (\varphi_p,\varphi_q)$ as follows:
\begin{align*}
a_p \cdot_{\varphi}a_q  & = \varphi_q^{-1}(\varphi_p(a_p) \cdot \varphi_q(a_q)) \\
& = l(\varphi_p(a_p)\cdot (ka_q)) \\
& = lk \varphi_p(a_p) \cdot a_q = \varphi_p(a_p) \cdot a_q
\end{align*}
which is the same module structure as in $A^{(\varphi_p,id_{A_q})}$.
\end{proof}

\begin{cor} \label{cor:module_orbits_under_ap}
Let $A$ be an additive group of order $p^2q$ ($q > p+1$).

If two brace multiplications $\cdot_1,\cdot_2$ are defined on $A$ which coincide on $A_p$ then $\cdot_1,\cdot_2$ define isomorphic brace structures on $A$ iff there is a brace automorphism $\varphi: A_p \to A_p$ such that
\[
\varphi(a) \cdot_2 b = a \cdot_1 b.
\] 
\end{cor}

Therefore, in order to classify the possible semidirect products $A_p \ltimes A_q$ up to isomorphism, it suffices to determine the orbits of the action of $\Aut(A_p)$ on $\Hom(A_p^{\circ},A_q)$ given by
\begin{align*}
\Hom(A_p^{\circ},A_q) \times \Aut(A_p) & \to \Hom(A_p^{\circ},A_q) \\
(f,\varphi) & \mapsto f \circ \varphi.
\end{align*}

This result will be put to use in the following sections.

\section{The case \texorpdfstring{$p \neq 2$}{p != 2}} \label{sec:case_p_not_2}

\subsection{Preparations}

To avoid confusion with vectors, we will follow Bachiller's convention by writing the binomial coefficient \glqq $n$ over $k$\grqq\ as $C(n,k)$.

Furthermore, we will repeatedly use the easy to prove identity
\begin{equation} \label{eq:chu_vandermonde}
C(x+y,2) = C(x,2) + C(y,2) + xy.
\end{equation}

Our first step will consist giving an isomorphism between the adjoint groups of $B^{p^2}$, $B^{p,p}$ and the groups $\Z_{p^2}$ resp. $\Z_p^2$:

\begin{lem} \label{lem:explicit_isomorphisms_not_2}
\begin{enumerate}[a)]
\item The map
\begin{align*}
\gamma: B^{p^2} & \to \Z_{p^2} \\
x & \mapsto x - p C(x,2)
\end{align*}
is an isomorphism between the adjoint group $(B^{p^2})^{\circ}$ and $\Z_{p^2}$.
\item The map
\begin{align*}
\delta: B^{p,p} & \to \Z_p^2 \\
\begin{pmatrix}
x \\ y
\end{pmatrix} & \mapsto \begin{pmatrix}
x - C(y,2) \\ y
\end{pmatrix}
\end{align*}
is an isomorphism between the adjoint group $(B^{p,p})^{\circ}$ and $\Z_p^2$.
\end{enumerate}
\end{lem}

\begin{rema}
It should be noted that the occurrence of the binomial coefficient $C(x,2)$ is the crucial reason for the distinguation of the cases $p \neq 2$ and $p = 2$, for in the latter case the mappings given above will not even be well-defined.
\end{rema}

\begin{proof}
\begin{enumerate}[a)]
\item We have $\gamma^{-1}(x) = x + pC(x,2)$, which can be seen as follows:

$x \equiv y \mod p$ clearly implies $C(x,2) \equiv C(y,2) \mod p$.

Therefore $pC(x,2) = pC(x - pC(x,2))$ holds in $\Z_{p^2}$, implying that
\[
(x - pC(x,2)) + pC(x- pC(x,2),2) = x - pC(x,2) + pC(x,2) = x.
\]
We now show that $\gamma^{-1}$ is a homomorphism:
\begin{align*}
\gamma^{-1}(x_1 + x_2) & = x_1 + x_2 + pC(x_1+x_2,2) \\
& = x_1 + x_2 + pC(x_1,2) + pC(x_2,2) + px_1x_2 \\
& = x_1 + pC(x_1,2) + x_2 + pC(x_2,2) + p\left[ (x_1 + pC(x_1,2))(x_2 + pC(x_2,2))\right] \\
& = \gamma^{-1}(x_1) \circ \gamma^{-1}(x_2).
\end{align*}
\item It is easy to see that
\[
\delta^{-1}\begin{pmatrix}
x \\ y
\end{pmatrix} = \begin{pmatrix}
x + C(y,2) \\ y
\end{pmatrix}.
\]
Now we calculate
\begin{align*}
\delta^{-1} \begin{pmatrix}
x_1 + x_2 \\ y_1 + y_2
\end{pmatrix} & = \begin{pmatrix}
 x_1 + x_2 + C(y_1+y_2,2) \\ y_1 + y_2
 \end{pmatrix} \\
 & = \begin{pmatrix}
 x_1 + x_2 + C(y_1,2) + C(y_2,2) + y_1y_2 \\ y_1 + y_2
 \end{pmatrix} \\
& = \begin{pmatrix}
x_1 + C(y_1,2) \\ y_1
\end{pmatrix} + \begin{pmatrix}
x_2 + C(y_2,2) \\ y_2
\end{pmatrix} + \begin{pmatrix}
y_1 y_2 \\ 0
\end{pmatrix} \\
& = \begin{pmatrix}
x_1 + C(y_1,2) \\ y_1
\end{pmatrix} \circ \begin{pmatrix}
x_2 + C(y_2,2) \\ y_2
\end{pmatrix} \\
& = \delta^{-1}\begin{pmatrix}
x_1 \\ y_1
\end{pmatrix} \circ
\delta^{-1} \begin{pmatrix}
x_2 \\ y_2
\end{pmatrix}
\end{align*}
\end{enumerate}
\end{proof}

Using these isomorphisms, we can make the action of $\Aut(A)$ on $\Aut(A^{\circ})$ visible when $A$ is a brace of order $p^2$.

We use the abbreviation ${}^fg := f g f^{-1}$ when $f: A \to B$, $g: B \to B$ are maps and $f$ is bijective.

\begin{lem} \label{lem:translating_automorphism_groups_not_2}
With the isomorphisms $\gamma, \delta$ from \autoref{lem:explicit_isomorphisms_not_2} holds
\begin{enumerate}[a)]
\item ${}^{\gamma} \Aut(B^{p^2}) = 1 + p\Z_{p^2} \leq \Z_{p^2}^{\times}$,
\item ${}^{\delta} \Aut(B^{p,p}) = \left\{ 
\begin{pmatrix}
a & b \\ 0 & d
\end{pmatrix} \vline a = d^2
\right\} \leq \Aut(\Z_p^2)$.
\end{enumerate}
\end{lem}

\begin{proof}
We have already determined the automorphism groups of the braces of order $p^2$ in \autoref{lem:automorphism_groups_p2}.

So it remains to translate them via $\gamma, \delta$:
\begin{enumerate}[a)]
\item In this case we do not need to calculate anything:

$\Aut(B^{p^2})$ has been determined as $1+p\Z_{p^2}$. These are exactly the elements in $\Aut\Z_{p^2}$ of order dividing $p$.

Applying the automorphism $\varphi \mapsto {}^{\gamma}\varphi$ will not alter this property - it follows that ${}^{\gamma}\Aut(B^{p^2}) \leq \Z_{p^2}^{\times}$ consists of $p$ elements of order dividing $p$ and must therefore coincide with $1 + p\Z_{p^2}$.
\item Here a calculation is necessary:

Applying $\varphi \mapsto {}^{\delta}\varphi$ to an automorphism of the form
$\varphi \begin{pmatrix}
x \\ y
\end{pmatrix} =
\begin{pmatrix}
a & b \\ 0 & d
\end{pmatrix}
\begin{pmatrix}
x \\ y
\end{pmatrix}
$ (with $a = d^2$) gives us
\begin{align*}
{}^{\delta}\varphi \begin{pmatrix}
x \\ y
\end{pmatrix} & = (\delta \varphi \delta^{-1})
\begin{pmatrix}
x \\ y
\end{pmatrix} \\
& = \delta \left( \begin{pmatrix}
a & b \\ 0 & d
\end{pmatrix}
\begin{pmatrix}
x + C(y,2) \\ y
\end{pmatrix}
 \right) \\
& = \delta \begin{pmatrix}
ax + aC(y,2) + by \\ dy
\end{pmatrix} \\
& = \begin{pmatrix}
ax + by + aC(y,2) - C(dy,2) \\ dy
\end{pmatrix} \\
& = \begin{pmatrix}
ax + \left( b + \frac{1}{2}(d - a) \right)y \\ dy
\end{pmatrix} \quad (\textnormal{using }a = d^2) \\
& = \begin{pmatrix}
a & b + \frac{1}{2}(d - a) \\ 0 & d
\end{pmatrix}
\begin{pmatrix}
x \\ y
\end{pmatrix}.
\end{align*}
For $b$ can take any value in $\Z_p$ we conclude that ${}^{\delta}\varphi$ again runs through all matrices of the form $\begin{pmatrix}
a & b^{\prime} \\ 0 & d
\end{pmatrix}$ with $a = d^2$.
\end{enumerate}
\end{proof}

Making $\Z_q$ an $A_p$-module is the same as giving an element $\varphi \in \Hom(A_p^{\circ},\Z_q^{\times})$, so that the module structure is given by $x \cdot a = \varphi(x)a - a$.

Furthermore, an element of $\Hom(\Z_{p^2},\Z_q^{\times})$ is given by $x \mapsto \omega^x$ where $\omega \in \Z_q^{\times}$ fulfils $\omega^{p^2}$. 

On the other hand, an element of $\Hom(\Z_p^2,\Z_q^{\times})$ can be written as $(x,y) \mapsto \omega^x \mu^y$ where $\omega,\mu \in \Z_q^{\times}$ fulfil $\omega^p = \mu^p = 1$.

With \autoref{prop:modules_are_circ_modules} and using the explicit isomorphisms of \autoref{lem:explicit_isomorphisms_not_2} we can make all these module structures explicit.

By \autoref{thm:p2q_braces_are_semidirect}, the resulting semidirect products make up the entirety of braces of order $p^2q$ ($q > p+1$).

This results in a coarse classification theorem which we state while at the same time introducing a nomenclature for the corresponding braces:

\begin{thm}[Coarse classification] \label{thm:coarse_classification_not_2}
The braces of order $p^2q$ with $q > p+1$  and $p \neq 2$ are given by
\begin{enumerate}[i)]
\item $T^{p,p,q}_{\omega,\mu} = \Z_p^2 \oplus \Z_q$, with multiplication
\[
\left( \begin{pmatrix}
x_1 \\ y_1
\end{pmatrix},a_1 \right) \cdot
\left( \begin{pmatrix}
x_2 \\ y_2
\end{pmatrix},a_2 \right) =
\left(
\begin{pmatrix}
0 \\ 0
\end{pmatrix},
(\omega^{x_1}\mu^{x_2}-1)a_2 \right)
\]
with $\omega,\mu \in \Z_q^{\times}$ fulfilling $\omega^p = \mu^p = 1$.
\item $T^{p^2,q}_{\omega} = \Z_{p^2} \oplus Z_q$, with multiplication
\[
(x_1,a_1) \cdot (x_2,a_2) = \left( 0, (\omega^{x_1}-1)a_2 \right)
\]
with $\omega \in \Z_q^{\times}$ fulfilling $\omega^{p^2} = 1$.
\item $B^{p,p,q}_{\omega,\mu} = \Z_p^2 \oplus \Z_q$, with multiplication
\[
\left( \begin{pmatrix}
x_1 \\ y_1
\end{pmatrix},a_1 \right) \cdot
\left( \begin{pmatrix}
x_2 \\ y_2
\end{pmatrix},a_2 \right) =
\left(
\begin{pmatrix}
y_1 y_2 \\ 0
\end{pmatrix},
(\omega^{x-C(y,2)} \mu^{y} -1)a_2 \right)
\]
with $\omega,\mu \in \Z_q^{\times}$ fulfilling $\omega^p = \mu^p = 1$.
\item $B^{p^2,q}_{\omega} = \Z_{p^2} \oplus \Z_q$, with multiplication
\[
(x_1,a_1) \cdot (x_2,a_2) = \left( px_1x_2, (\omega^{x_1-pC(x_1,2)}-1)a_2 \right)
\]
with $\omega \in \Z_q^{\times}$ fulfilling $\omega^{p^2} = 1$.
\end{enumerate}
\end{thm}

Now remains the task of singling out repetitions in this list, i.e. of finding one representative for each isomorphism class only.

Following \autoref{cor:module_orbits_under_ap} it suffices to determine the orbits of the natural action of $\Aut(A_p)$ on $\Hom(A_p^{\circ},\Z_q^{\times})$.

This method allows us to find non-isomorphic representants for the isomorphism classes of the braces described above. This task will be treated in the following subsections.

\subsection{\texorpdfstring{$p \nmid q-1$}{p -| q-1}:} \label{subsec:p_part_1}~

For $q-1 = \vert \Z_q^{\times} \vert$ is not divisible by $p$, we have $\Hom(A_p^{\circ},\Z_q^{\times}) = \{ 1 \}$ in every case.

We could alternatively argue that we can only assign $\omega,\mu = 1$. Consequently, the only module structure on $\Z_q$ is the trivial one.

Therefore, for $p \nmid q-1$, only the braces $T_{1,1}^{p,p,q}$, $T_{1}^{p^2,q}$, $B_{1,1}^{p,p,q}$, $B_{1}^{p^2,q}$ can arise.

\subsection{\texorpdfstring{$p \mid q-1, p^2 \nmid q-1$}{p | q-1, p\texttwosuperior\ -| q-1}:} \label{subsec:p_part_2}~

In this case, the image of any homomorphism $\varphi: A_p^{\circ} \to \Z_q^{\times}$ must consist of elements of order $p$.

Using that these elements make up a group isomorphic to $\Z_p$ we must therefore determine the orbits of $\Hom(A_p^{\circ},\Z_p)$ under the action of $\Aut(A_p)$:

\begin{enumerate}[a)]
\item $A_p \cong T^{p,p}$:

Here we have that $\Aut(A_p)$ is the full automorphism group of $A_p^{\circ} = \Z_p^2$. We must therefore determine the orbits of $\Hom(\Z_p^2,\Z_p)$ under $\Aut(\Z_p^2)$.

Writing the elements of $\Hom(\Z_p^2,\Z_p)$ and $\Aut(\Z_p^2)$ as matrices $(x \ y)$ resp. $\begin{pmatrix}
a & b \\ c & d
\end{pmatrix}$, the latter being invertible, this action is represented by matrix multiplications from the right.

Therefore there are only two orbits: the orbit represented by the zero homomorphism and the orbit represented by any non-zero homomorphism, say, $(1 \ 0)$.

This results in the braces $T^{p,p}_{1,1}$ and $T^{p,p}_{\omega,1}$ where $\omega$ is a fixed $p$-th root of unity in $\Z_{q}^{\times}$.

\item $A_p \cong T^{p^2}$:

$\Aut(A_p)$ is the full automorphism group of $A_p^{\circ} = \Z_{p^2}$, i.e. $\Z_{p^2}^{\times}$:

Any element of $\Hom(\Z_{p^2},\Z_p)$ is given by sending the generator $1$ to an element $a \in \Z_p$. Call this homomorphism $\varphi_a$

Applying an $x \in \Z_{p^2}^{\times}$ to this element gives us the homomorphism $\varphi_{\overline{x}a}$ where $\overline{x} \in \Z_p^{\times}$ is the reduction of $x$ modulo $p$. The reduction map is clearly a surjection $\Z_{p^2}^{\times} \to \Z_p^{\times}$, therefore $\varphi_{a_1}$ and $\varphi_{a_2}$ lie in the same orbit iff $a_1 = a_2 = 0$ or $a_1,a_2 \neq 0$.

This results in the two braces $T^{p^2}_{1}$ and $T^{p^2}_{\omega}$ where $\omega \in \Z_q^{\times}$ is a fixed $p$-th root of unity.

\item $A_p \cong B^{p,p}$:

Writing the action as matrix multiplications, the action is given by right-multipli\-cation of vectors $(x \ y)$ by invertible matrices $\begin{pmatrix}
a & b \\ 0 & d
\end{pmatrix}$ ($a = d^2$).

It is easy to see that there is one orbit consisting solely of $(0 \ 0)$.

The orbit of $(0 \ 1)$ is made up from all vectors $(0 \ d)$ with $d \neq 0$.

For $x \neq 0$, the orbit of $(x \ 0)$ consists of all vectors of the form $(ax \ b) = (d^2 x \ b)$ with $d \neq 0$.

$b$ can be arbitrarily chosen but it becomes apparent here that two vectors $(x_1 \ y_1),\allowbreak (x_2 \ y_2)$ with $x_1,x_2 \neq 1$ lie in the same orbit iff $\frac{x_1}{x_2}$ is of the form $d^2$, i.e. iff both of $x_{1,2}$ are quadratic residues resp. non-residues in $\Z_p^{\times}$.

This results in $4$ classes of braces:

$B^{p,p}_{1,1}$, $B^{p,p}_{1,\omega}$, $B^{p,p}_{\omega,1}$, $B^{p,p}_{\omega^{\eta},1}$ where $\omega \in \Z_q^{\times}$ is a $p$-th root of unity and $\eta \in \Z_p^{\times}$ is a quadratic non-residue.
\item $A_p \cong B^{p^2}$:

Identifying $(B^{p^2})^{\circ} \cong \Z_{p^2}$, the action on $\Hom(\Z_{p^2},\Z_p)$ is given by multiplying the argument with elements $x \in 1 + p\Z_{p^2}$. As above, we write the homomorphism sending $1$ to $a \in \Z_p$ as $\varphi_a$. Under the action of $x$, however, $\varphi_a$ is mapped to $\varphi_{\overline{x}a} = \varphi_{a}$, i.e. there is no non-trivial action.

The orbits are therefore given by the single elements of $\Hom(\Z_{p^2},\Z_p)$, i.e. $\varphi_0,\varphi_1,\ldots,\allowbreak \varphi_{p-1}$, resulting in the following $p$ braces:

$B^{p,p}_{1}$ and the $p-1$ braces $B^{p,p}_{\omega^i}$ ($i = 1,\ldots,p-1$) where $\omega \in \Z_q^{\times}$ is a fixed $p$-th root of unity.
\end{enumerate}

\subsection{\texorpdfstring{$p^2 \mid q-1$}{p \texttwosuperior\ | q-1}:} \label{subsec:p_part_3}~

The image of a homomorphism $\varphi: A_p^{\circ} \to \Z_q^{\times}$ now consists of elements of order dividing $p^2$.

These elements make up a subgroup of $\Z_q^{\times}$ isomorphic to $\Z_{p^2}$, so the problem is reduced to the determination of the orbits of $\Hom(A_p^{\circ},\Z_{p^2})$ under the action of $\Aut(A_p)$.

In the cases $A_p \cong T^{p,p}, B^{p,p}$, no new braces can arise, for $A_p^{\circ}$ does not have any elements of order $p^2$.

Therefore, we have to concentrate on $T^{p^2}$ and $B^{p^2}$:

\begin{enumerate}[a)]
\item $A_p \cong T^{p^2}$:

Clearly, $\Aut(T^{p^2})$ is the full automorphism group of $(T^{p^2})^{\circ} = \Z_{p^2}^{\times}$ and can be identified with $\Z_{p^2}^{\times}$.

Each element of $\Hom(\Z_{p^2},\Z_{p^2})$ can be represented by a multiplication with an $a \in \Z_{p^2}$.

The action is therefore represented by multiplying arbitrary elements of $\Z_{p^2}$ by units in $\Z_{p^2}$. This implies that $a_{1,2}$ are in the same orbit iff they generated the same subgroup in $\Z_{p^2}$.

Therefore, representants of the different orbits may given by $a= 0,1,p$, resulting in three braces:

$T^{p^2,q}_{1}$, $T^{p^2,q}_{\omega}$, and $T^{p^2,q}_{\tilde{\omega}}$ where $\omega,\tilde{\omega}$ is a fixed $p$-th respective $p^2$-th root of unity in $\Z_q^{\times}$.
\item $A_p \cong B^{p^2}$:

We set up the analysis as in the case $T^{p^2}$ but now the action is represented by multiplying arbitrary $a \in \Z_{p^2}$ by units in $1+p\Z_{p^2}$.

If $a_{1,2}$ are units themselves then they lie in the same orbit iff $\frac{a_1}{a_2} \in 1 + p\Z_{p^2}$, i.e. iff $a_1 - a_2 \in p\Z_{p^2}$. This results in $p-1$ orbits which may be represented by $1,\ldots,p-1$.

If $a = pj$ ($j \in \Z_{p^2}$) it is fixed under the action of $1 + p\Z_{p^2}$, as can be seen from the calculation
\[
pj (1+ pk) = pj + p^2(jk) = pj.
\]
Therefore, any such $a$ makes up a single orbit and we can represent these by the $p$ elements $0,p,2p,\ldots,(p-1)p$.

This results in the following $2p-1$ braces:

$B^{p^2}_{1}$, $B^{p^2}_{\omega^i}$, $B^{p^2}_{\tilde{\omega}^j}$ with $\omega,\tilde{\omega}$ as above and $i,j = 1,\ldots p-1$.
\end{enumerate}

\subsection{Conclusion}

\begin{thm} [Fine classification] \label{thm:fine_classification_p_not_2}
Every brace of order $p^2q$ with $2 < p < q$ is isomorphic to \emph{exactly} one of the following braces:
\begin{itemize}
\item \begin{enumerate}[a)]
\item $T^{p,p,q}_{1,1}$,
\item $T^{p^2,q}_{1}$
\item $B^{p,p,q}_{1,1}$,
\item $B^{p^2,q}_{1}$,
\end{enumerate}
\item if $p \mid q-1$, additionally one of
\begin{enumerate}[a)] \setcounter{enumi}{5}
\item $T^{p,p,q}_{\omega,1}$,
\item $T^{p^2,q}_{\omega}$,
\item the three braces $B^{p,p}_{1,\omega}$, $B^{p,p}_{\omega,1}$ $B^{p,p}_{\omega^{\eta},1}$
\item the $p-1$ braces $B^{p^2}_{\omega^i}$ ($i = 1,\ldots,p-1$)
\end{enumerate}
where $\omega$ is a fixed $p$-th root of unity in $\Z_q^{\times}$ and $\eta$ is a fixed quadratic non-residue in $\Z_p^{\times}$.
\item 
additionally, if $p^2 \mid q-1$, one of
\begin{enumerate}[a)] \setcounter{enumi}{9}
\item $T^{p^2,q}_{\tilde{\omega}}$,
\item the $p-1$ braces $T^{p^2,q}_{\tilde{\omega}^i}$ ($i = 1,\ldots,p-1$)
\end{enumerate}
where $\tilde{\omega}$ is a fixed $p^2$-th root of unity in $\Z_{q}^{\times}$.
\end{itemize}
\end{thm}

Recall that $b(n)$ was defined as the number of braces of order $n$.

Summing everything up we get the result:

\begin{cor} \label{cor:counting_p_not_2}
For primes $2 < p < q$ holds
\[
b(p^2q) = \begin{cases}
4 & p \nmid q-1 \\
p + 8 & p \mid q-1,p^2 \nmid q-1 \\
2p + 8 & p^2 \mid q-1.
\end{cases}
\]
\end{cor}

\begin{rema}
For $p \nmid q-1$, \autoref{cor:counting_p_not_2} has already been proved by Smoktunowicz in \cite[p.6]{Smoktunowicz_note_on_sybe}. She shows that in the respective semidirect decomposition $A = A_p \ltimes A_q$ the component $A_q$ must be a trivial $A_p$-module, implying that $A$ is a direct product $A_p \times A_q$ with no interplay between $A_p$ and $A_q$. $A_q$, being of prime order, must therefore be trivial, and $A_p$ must be one of the four braces of \autoref{thm:bachiller_classification}, therefore verifying the conjecture of Guarnieri and Vendramin regarding this case.

Essentially, in this case, our argument runs along the same lines as the one given by Smoktunowicz.
\end{rema}

\section{The case \texorpdfstring{$p = 2$}{p = 2}} \label{sec_case_p_2}

\subsection{Preparations}

In this section, we will sometimes write the elements of $\Z_4$ as their corresponding binary strings, i.e. $0 = \underline{00}$, $1 = \underline{01}$, $2 = \underline{10}$ and $3 = \underline{11}$. The reason will soon become apparent.

Again, we start our analysis by making the groups $A_2^{\circ}$ explicit:

\begin{lem} \label{lem:explicit_isomorphisms_2}
\begin{enumerate}[a)]
\item The map
\begin{align*}
\alpha: B^{2,2} & \to \Z_4 \\
\begin{pmatrix}
x \\ y
\end{pmatrix} & \mapsto \underline{xy}
\end{align*}
gives an isomorphism between $(B^{2,2})^{\circ}$ and $\Z_4$.
\item The map
\begin{align*}
\beta: B^{4} & \to \Z_2^2 \\
\underline{xy} & \mapsto \begin{pmatrix}
x \\ y
\end{pmatrix}
\end{align*}
is an isomorphism between $(B^4)^{\circ}$ and $\Z_2^2$.
\end{enumerate}
\end{lem}

\begin{proof}
This is a long but easy calculation.
\end{proof}

Again we express the action of the (brace) automorphism groups determined in \autoref{lem:automorphism_groups_p2} on $A_2^{\circ}$ in terms of these explicit isomorphisms:

\begin{lem} \label{lem:translating_automorphism_groups_2}
\begin{enumerate}[a)]
\item $^{\alpha}\Aut(B^{2,2}) = \{ \pm 1 \} = \Z_4^{\times}$
\item $^{\beta}\Aut(B^4) = \left\{ \begin{pmatrix}
1 & a \\ 0 & 1
\end{pmatrix} \vline a \in \Z_2 \right\} \leq \Aut(\Z_2^2)$.
\end{enumerate}
\end{lem}

\begin{proof}
\begin{enumerate}[a)]
\item \autoref{lem:automorphism_groups_p2} shows that $B^{2,2}$ has exactly two automorphisms. $\Aut(\Z_4) \allowbreak = \{ \pm 1 \}$ also is of order $2$, thus proving the claim.
\item \autoref{lem:automorphism_groups_p2} shows that the only nontrivial automorphism of $B^4$ is given by $x \mapsto 3x$. It sends $1 = \underline{01}$ to $3 = \underline{11}$ and leaves $2 = \underline{10}$ fixed.

The respective automorphism of $\Aut(\Z_2^2)$ must then send $\begin{pmatrix} 0 \\ 1 \end{pmatrix} $ to $\begin{pmatrix} 1 \\ 1 \end{pmatrix} $ and leave $\begin{pmatrix} 1 \\ 0 \end{pmatrix} $ fixed and therefore is given by the matrix $\begin{pmatrix} 1 & 1\\ 0 & 1 \end{pmatrix} $.
\end{enumerate}
\end{proof}

Using this identification and repeating the arguments preceeding \autoref{thm:coarse_classification_not_2} we directly get the

\begin{thm}[Coarse classification] \label{thm:coarse_classification_2}
The braces of order $4q$ with $q > 3$ are given by
\begin{enumerate}[i)]
\item $T^{2,2,q}_{\omega,\mu} = \Z_2^2 \oplus \Z_q$, with multiplication
\[
\left( \begin{pmatrix}
x_1 \\ y_1
\end{pmatrix},a_1 \right) \cdot
\left( \begin{pmatrix}
x_2 \\ y_2
\end{pmatrix},a_2 \right) =
\left(
\begin{pmatrix}
0 \\ 0
\end{pmatrix},
(\omega^{x_1}\mu^{x_2}-1)a_2 \right)
\]
with $\omega,\mu \in \{ \pm 1 \} \leq \Z_q^{\times}$.
\item $T^{4,q}_{\omega} = \Z_4 \oplus \Z_q$, with multiplication
\[
(x_1,a_1) \cdot (x_2,a_2) = \left( 0, (\omega^{x_1}-1)a_2 \right)
\]
with $\omega \in \Z_q^{\times}$ fulfilling $\omega^4 = 1$,
\item $B^{2,2,q}_{\omega} = \Z_2^2 \oplus \Z_q$, with multiplication
\[
\left( \begin{pmatrix}
x_1 \\ y_1
\end{pmatrix},a_1 \right) \cdot
\left( \begin{pmatrix}
x_2 \\ y_2
\end{pmatrix},a_2 \right) = 
\left(
\begin{pmatrix}
y_1 y_2 \\ 0
\end{pmatrix},
(\omega^{\underline{x_1y_1}} -1)a_2 \right),
\]
with $\omega \in \Z_q^{\times}$ fulfilling $\omega^4 = 1$,
\item $B^{4,q}_{\omega,\mu} = \Z_4 \oplus \Z_q$, with multiplication
\[
(\underline{x_1y_1},a_1) \cdot (\underline{x_2y_2},a_2) = \left(2 \ \underline{x_1y_1} \ \underline{x_2y_2} \ , (\omega^{x_1} \mu^{y_1} -1)a_2 \right),
\]
with $\omega,\mu \in \{ \pm 1 \} \leq \Z_q^{\times}$.
\end{enumerate}
\end{thm}

Again we continue with determining non-isomorphic representants for the isomorphism classes of the braces described above.

\subsection{\texorpdfstring{$4 \nmid q-1$}{4 -| q-1}:}~

In this case the image of $\Hom(A_2^{\circ},\Z_q)$ can only consist of the elements of order dividing $2$, i.e. of $\pm 1$. We will identify this subgroup with $\Z_2$.

\begin{enumerate}[a)]
\item $A_2 \cong T^{2,2},T^4$:

This can be subsumed under the discussions of the general cases $T^{p,p}$ and $T^{p^2}$ from \autoref{subsec:p_part_2}, resulting in the four non-isomorphic braces $T^{2,2,q}_{\pm 1,1}$ and $T^{4,q}_{\pm 1}$.

\item $A_2 \cong B^{2,2}$:

Under the isomorphism $(B^{2,2})^{\circ} \cong \Z_4$, the action on $\Hom(\Z_4,\Z_2)$ consists of changing signs in the argument. But $\varphi(\pm x) = \pm \varphi(x) = \varphi(x)$ for the values of $\varphi$ lie in $\Z_2$.

\item $A_2 \cong B^4$:

Using the isomorphism $(B^4)^{\circ} \cong \Z_2^2$, the action on $\Hom(\Z_2^2,\Z_2)$ is represented by right-multiplying row vectors $(x \ y)$ by upper unitriangular matrices.

Under this action $\Hom(\Z_2^2,\Z_2)$ decomposes as a disjoint union of the orbits of $(0 \ 0)$, $(0 \ 1)$ and $(1 \ 0)$, resulting in the three braces $B^4_{\pm 1,1}$ and $B^4_{1,-1}$.
\end{enumerate}

\subsection{\texorpdfstring{$4 \mid q-1$}{4 | q-1}:}~

Now the image may additionally contain elements of order $4$ which however only arise when $A_2^{\circ}$ contains elements of order $4$ - which make up a subgroup isomorphic to $\Z_4$ in $\Z_q^{\times}$.

This will only happen for $T^4$ and $B^{2,2}$.

We must now determine the orbits of $\Hom(A_2^{\circ},\Z_4)$ under the action of $\Aut(A_2)$:

\begin{enumerate}[a)]
\item $A_2 \cong T^4$:

The proof of this case is contained in \autoref{subsec:p_part_3} for this is a special case of the more general case $A_p \cong T^{p^2}$, resulting in the three braces $T^{4}_{\pm 1}, T^{4,q}_{\omega}$ where $\omega \in \Z_q^{\times}$ is a fixed $4$-th root of unity. 
\item $A_2 \cong B^{2,2}$:

Identifying $(B^{2,2})^{\circ}$ with $\Z_4$, the action on $\Hom(\Z_4,\Z_4)$ just consists of changing the sign of a homomorphism.

Representing the elements of $\Hom(\Z_4,\Z_4)$ as multiplications by elements of $\Z_4$, the orbits become therefore represented by the elements $0,1,2$, leading us to the following three braces:

$B^{2,2,q}_{\pm 1}$ and $B^{2,2,q}_{\omega}$ with $\omega \in \Z_q^{\times}$ being a fixed $4$-th root of unity.
\end{enumerate}

\subsection{Conclusion}

\begin{thm}[Fine classification] \label{thm:fine_classification_p_2}
Every brace of order $4q$ with $q > 3$ is isomorphic to \emph{exactly} one of the following braces:
\begin{itemize}
\item
\begin{enumerate}[a)]
\item the two braces $T^{2,2,q}_{\pm 1,1}$,
\item the two braces $T^{4,q}_{\pm 1}$,
\item the two braces $B^{2,2,q}_{\pm 1}$,
\item the three braces $B^{4,q}_{1,\pm 1}$, $B^{4,q}_{-1,1}$.
\end{enumerate}

\item if $4 \mid q-1$, additionally one of
\begin{enumerate}[a)]\setcounter{enumi}{4}
\item $T^{4,q}_{\omega}$
\item $B^{2,2,q}_{\omega}$
\end{enumerate}
where $\omega \in \Z_q^{\times}$ is a fixed $4$-th root of unity.
\end{itemize}
\end{thm}

Summing everything up, we get the result

\begin{cor} \label{cor:counting_p_2}
For $q > 3$ prime holds
\[
b(4q) = \begin{cases}
9 & 4 \nmid q-1 \\
11 & 4 \mid q-1.
\end{cases}
\]
\end{cor}

\section*{Acknowledgements}

I would like to express my deep gratitude to Wolfgang Rump and Leandro Vendramin for valuable suggestions regarding the quality of the presentation.

\bibliographystyle{halpha}

\end{document}